\documentclass[11pt]{article}
\usepackage[colorlinks=false]{hyperref}
\usepackage{breakurl} 
\usepackage{amsfonts, amsmath, amssymb, amsthm}
\usepackage{enumitem, tikz, multirow}

\definecolor{darkgray}{RGB}{64,64,64}
\definecolor{litegray}{RGB}{192,192,192}
\tikzstyle{vertex}=[circle, draw, fill=litegray, inner sep=0pt, minimum width=4pt]

\usepackage{mathtools}
\DeclarePairedDelimiter{\floor}{\lfloor}{\rfloor}

\author{Zilin Jiang\thanks{Department of Mathematics, Massachusetts Institute of Technology, Cambridge, MA 02139, USA. Email: {\tt zilinj@mit.edu}. The work was done when Z. Jiang was a postdoctoral fellow at Technion -- Israel Institute of Technology, and was supported in part by Israel Science Foundation (ISF) grant nos.\ 1162/15, 936/16.} \and Nikita Polyanskii\thanks{CDISE, Skolkovo Institute of Science and Technology, and Department of Mathematics, Technion – Israel Institute of Technology. Email: {\tt nikita.polyansky@gmail.com}. Supported in part by ISF grant nos 1162/15, 326/17, and by the Russian Foundation for Basic Research through grant nos.\ 16-01-00440~A, 18-07-01427~A, 18-31-00310~MOL\_A.}}

\title{On the metric dimension of Cartesian powers of a graph}
\date{}

\newtheorem{theorem}{Theorem}
\newtheorem{corollary}[theorem]{Corollary}
\newtheorem{conjecture}{Conjecture}

\newtheorem{lemma}[theorem]{Lemma}
\newtheorem{proposition}[theorem]{Proposition}

\theoremstyle{definition}
\newtheorem{definition}{Definition}

\theoremstyle{remark}
\newtheorem{remark}{Remark}

\newcommand{\abs}[1]{\left\lvert#1\right\rvert}
\newcommand{\dset}[2]{\left\{#1 : #2\right\}}
\newcommand{\sset}[1]{\left\{#1\right\}}

\newcommand{\from}{\colon}
\newcommand{\meet}{\wedge}
\newcommand{\gsn}{G^{\square n}}
\newcommand{\one}{\mathbf{1}}
\newcommand{\zero}{\mathbf{0}}

\newcommand{\Z}{\mathbb{Z}}
\newcommand{\N}{\mathbb{N}}
\newcommand{\Q}{\mathbb{Q}}

\newcommand{\EE}[1]{\mathrm{E}\left[{#1}\right]}
\newcommand{\OO}[1]{O\left({#1}\right)}
\newcommand{\pr}[1]{\operatorname{Pr}\left({#1}\right)}

\begin{document}

\maketitle

\begin{abstract}
  A set of vertices $S$ resolves a graph if every vertex is uniquely determined by its vector of distances to the vertices in $S$. The metric dimension of a graph is the minimum cardinality of a resolving set of the graph. Fix a connected graph $G$ on $q \ge 2$ vertices, and let $M$ be the distance matrix of $G$. We prove that if there exists $w \in \mathbb{Z}^q$ such that $\sum_i w_i = 0$ and the vector $Mw$, after sorting its coordinates, is an arithmetic progression with nonzero common difference, then the metric dimension of the Cartesian product of $n$ copies of $G$ is $(2+o(1))n/\log_q n$. In the special case that $G$ is a complete graph, our results close the gap between the lower bound attributed to Erd\H{o}s and R\'enyi and the upper bounds developed subsequently by Lindstr\"om, Chv\'atal, Kabatianski, Lebedev and Thorpe.
\end{abstract}

\section{Introduction}\label{intro}

A set of vertices $S$ \emph{resolves} a graph if every vertex is uniquely determined by its vector of distances to the vertices in $S$. The \emph{metric dimension} of a graph is the minimum cardinality of a resolving set of the graph. The \emph{Cartesian product} of graphs $G_1, \dots, G_n$ is the graph with vertex set $V(G_1) \times \dots \times V(G_n)$ such that $(u_1, \dots, u_n)$ and $(v_1, \dots, v_n)$ are adjacent whenever there exists $j\in [n]$ such that $u_i = v_i$ for all $i \neq j$ and $u_j$ is adjacent to $v_j$ in $G_j$.

For a graph $G$ and $n\in \N$, denote by $\gsn$ the Cartesian product of $n$ copies of $G$, and by $m(G, n)$ the metric dimension of $\gsn$. This paper undertakes the study of the asymptotic behavior of $m(G, n)$ when the connected graph $G$ is fixed and $n$ tends to infinity, especially when $G$ is a complete graph on $q$ vertices, which we denote by $K_q$. In this context, the definition of a resolving set can be rephrased in the following way. Denote the distance between vertices $u, v$ in $G$ by $d(u, v)$. Given a subset $S$ of $V^n$, define $d_S\from V^n\to \N^S$ by $(d_S(v))_s = d(v_1, s_1) + \dots + d(v_n, s_n)$ for every $v=(v_1, \dots, v_n)\in V^n$ and $s = (s_1, \dots, s_n) \in S$. The set of vertices $S$ is a resolving set of $\gsn$ if and only if $d_S$ is an injection.

The concept of resolving set and that of metric dimension date back to the 1950s --- they were defined by Bluementhal~\cite{MR0054981} in the context of metric space. These notions were introduced to graph theory by Harary and Melter~\cite{MR0457289} and Slater~\cite{MR0422062} in the 1970s.

Under the guise of a coin weighing problem, the metric dimension of a hypercube was first studied by Erd\H{o}s and R\'enyi. The coin weighing problem, posed by S\"oderberg and Shapiro~\cite{MR1532427}, assumes $n$ coins of weight $a$ or $b$, where $a$ and $b$ are known, and an accurate scale. S\"oderberg and Shapiro asked the question of how many weighings are needed to determine which of $n$ coins are of weight $a$ and which of weight $b$ if the numbers of each are not known. The variant of the problem, where the family of weighings has to be given in advance, is connected to the metric dimension of the hypercube $K_2^{\square n}$. It was observed that the minimum number of weighings differs from $m(K_2, n)$ by at most $1$ (see \cite[Section 1]{MR2065985}). A lower bound on the number of weighings by Erd\H{o}s and R\'enyi~\cite{MR0165988} and an upper bound by Lindstr\"om~\cite{MR0168488} and independently by Cantor and Mills~\cite{MR0223248} imply that $m(K_2, n) = (2+o(1))n/\log_2 n$.

The metric dimension of the Hamming graph $K_q^{\square n}$ is also connected to the Mastermind game. Mastermind is a deductive game for two players, the codemaker and the codebreaker\footnote{In honor of Erd\H{o}s, Chv\'atal~\cite{MR729785} referred to the codemaker and the codebreaker as SF and PGOM. See \cite[p.~41 and p.~70]{MR1638921} for what SF and PGOM stand for.}. In this game, the codemaker conceals a vector $x = (x_1, \dots, x_n) \in [q]^n$, and the codebreaker, who knows both $q$ and $n$, tries to identify $x$ by asking a number of questions, which are answered by the codemaker.
Each question is a vector $y = (y_1, \dots, y_n) \in [q]^n$; each answer consists of a pair of numbers $a(x,y)$, the number of subscripts $i$ such that $x_i = y_i$, and $b(x,y)$, the maximum number of $a(x,\tilde{y})$ with $\tilde{y}$ running through all the permutations of $y$. Knuth~\cite{MR0434680} has shown that four questions suffice to determine $x$ in the commercial version of the game where $n = 4$ and $q = 6$. Suppose for the time being that we remove the second number $b(x,y)$ from the answers given by the codemaker and we require that the questions from the codebreaker are sent all at once. In this version of Mastermind, the minimum number of questions required to determine $x$ is exactly $m(K_q, n)$ (see~\cite[Section 6]{MR2318676}).

Kabatianski, Lebedev and Thorpe~\cite{kabatianski2000mastermind} stated that a straightforward generalization of the lower bound on $m(K_2, n)$ by Erd\H{o}s and R\'enyi~\cite{MR0165988} gives $m(K_q, n) \ge (2+o(1))n/\log_q n$. Kabatianski et al.\ also asserted that more precise calculations, based on the probabilistic method of Chv\'atal~\cite[Theorem 1]{MR729785}, would show that $m(K_q, n) \le (2+o(1))\log_q(1+(q-1)q) \cdot n/\log_q n$.
Very recently, these calculations were carried out by Kabatianski and Lebedev~\cite{kabatiansky2018metric}. Moreover, they proved that $m(K_q, n) = (2+o(1))n/\log_q n$ for $q = 3, 4$, which was previously announced in \cite[Theorem 1]{kabatianski2000mastermind}, and they conjectured that $m(K_q, n) = (2+o(1))n/\log_q n$ for all $q \ge 2$. We emphasize that the asymptotic behavior is different when $q$ varies and $n$ is fixed.
For example, C\'aceres et al.~\cite[Theorem 6.1]{MR2318676} showed that $m(K_q, 2) = \floor{2(2q-1)/3}$.

Motivated by the above applications, we establish an upper bound and a lower bound on $m(G, n)$ for every connected graph $G$ in Section~\ref{ub} and Section~\ref{lb} respectively. For certain families of graphs, the lower bound and the upper bound are asymptotically equivalent. In particular, we show that $m(K_q, n) = (2+o(1))n/\log_q n$ for all $q \ge 2$ in Section~\ref{tight}. We conclude with a generalization to integer matrices and some open problems in Section~\ref{concluding}.

\section{An upper bound on $m(G, n)$}\label{ub}

We establish the following upper bound on $m(G, n)$.

\begin{theorem}\label{ubt}
  Given a connected graph $G$ on $q \ge 2$ vertices, let $M$ be the distance matrix of $G$. For every $n\in \N$, the metric dimension $m(G, n)$ of $\gsn$ is at most $$\left(2+\OO{\frac{\log \log n}{\log n}}\right)\frac{n}{\log_{r}n},$$ where $r = r(G)$ is defined by
  \begin{equation}\label{parameter}
    r(G) = \min\dset{\frac{\max Mw - \min Mw}{\gcd_{i < j}((Mw)_i-(Mw)_j)}}{w\in \Z^q, \sum_i w_i = 0, (Mw)_i \neq (Mw)_j \text{ for all }i \neq j} + 1.
  \end{equation}
\end{theorem}

\begin{remark}\label{finite}
  Alternatively, $r(G)$ is the shortest length of an arithmetic progression with nonzero common difference that contains $Mw$, after sorting its coordinates, as a subsequence for some $w \in \Z^q$ such that $\sum_i w_i = 0$. Clearly, $r(G) \ge q$. It is less clear that $r(G) < \infty$. We claim that there exists $w\in \Q^q$ such that $\sum_i w_i = 0$ and $(Mw)_i \neq (Mw)_j$ for all $i \neq j$. Denote the $i$th row of $M$ by $M_i$. For $i \neq j$, because $(M_i - M_j)_i + (M_i - M_j)_j = 0$, the equation $(M_i-M_j)w = 0$, or $(Mw)_i = (Mw)_j$, defines a subspace of $\Q^n$ different from $\dset{w\in \Q^n}{\sum_i w_i = 0}$. In other words, $(Mw)_i = (Mw)_j$ defines a $1$-codimensional subspace of $\dset{w\in \Q^n}{\sum_i w_i = 0}$, and so $\dset{w\in \Q^n}{\sum_i w_i = 0}\setminus \cup_{i\neq j}\dset{w\in \Q^n}{(Mw)_i = (Mw)_j}$ is nonempty. Finally, we scale $w$ properly so that it becomes a vector in $\Z^q$.
\end{remark}

Our construction of a resolving set of $\gsn$ is inspired by the upper bound for the coin weighing problem by Lindstr\"om~\cite{MR0181604}. Among various constructions such as the recursive construction by Cantor and Mills~\cite{MR0223248} and the construction by Bshouty~\cite{bshouty2009optimal} based on Fourier transform, we find the one using the theory of M\"obius functions by Lindstr\"om~\cite{MR0316369} best suits our needs.

We recall the basics of M\"obius functions. Let $(P, \prec)$ be a locally finite partially ordered set. The M\"obius function $\mu\from P\times P \to \Z$ can be defined inductively by the following relation:
\[
  \mu(x,y) = \begin{cases}
    1 & \text{if }x = y, \\
    -\sum_{x\preceq z \prec y}\mu(x, z) & \text{for }x \prec y, \\
    0 & \text{otherwise}.
  \end{cases}
\]
The classical M\"obius function in number theory is essentially the M\"obius function of the set of natural numbers $\N = \sset{0, 1, \dots}$ partially ordered by divisibility. For our purpose, we first consider binary representation of natural numbers, and we instead partially order $\N$ in the following way: $x \preceq y$ if and only if $x = x \meet y$, where $\meet$ is the bitwise AND operation\footnote{A bitwise AND takes two binary representations and perform the logical AND operation on each pair of the corresponding bits. For example, $6 \meet 11 = 0110_2 \meet 1011_2 = 0010_2 = 2$.}. The M\"obius function is thus $$\mu(x,y) = (-1)^{n(x)-n(y)}, \text{if }x\preceq y,$$ where $n(x)$ is the number of ones in the binary representation of $x$. With the binary operator $\meet$, the partially ordered set $(\N, \prec)$ is indeed a \emph{meet-semilattice} --- a partially ordered set in which any pair of elements has the greatest lower bound. We need the following identity for our meet-semilattice.

\begin{lemma}[Lemma of Lindstr\"om~\cite{MR0238738}]\label{mobius}
  Let $(P, \prec, \meet)$ be a locally finite meet-semilattice with M\"obius function $\mu(x,y)$. Let $a, b\in P$ and $b \npreceq a$. Let $f(x)$ be defined for all $x\preceq a\meet b$ with values in a commutative ring with unity. Then we have \[
    \sum_{x \preceq b} f(x\meet a)\mu(x, b) = 0.
  \]
\end{lemma}

The last ingredient is the following estimation on the partial sum of $n(\cdot)$.

\begin{theorem}[Theorem 1 of Bellman and Shapiro~\cite{MR0023864}]\label{estimate_ns}
  $$\sum_{i = 0}^x n(i) = \tfrac{1}{2}x\log_2 x + O(x\log \log x) \text{ as }x\to\infty.$$
\end{theorem}

We now construct a resolving set for Theorem~\ref{ubt} using the M\"obius function of $(\N, \prec, \meet)$.

\begin{proof}[Proof of Theorem~\ref{ubt}]
  Let $w\in \Z^q$ be such that $\sum_i w_i = 0$ and the coordinates of $Mw$ are distinct integers such that
  \begin{equation}\label{r_condition}
    r = r(G) > \frac{\max Mw - \min Mw}{\gcd_{i < j}((Mw)_i - (Mw)_j)}.
  \end{equation}
  Set $\abs{w}_1 := \sum_i \abs{w_i}$.
  For each $j \in \N$, let $b(j)$ be the largest integer such that
  \begin{equation} \label{eqbj}
    r^{b(j)} \cdot \abs{w}_1 \le 2^{n(j)},
  \end{equation}
  that is, $b(j) := \floor{n(j)\log_r 2 - \log_r \abs{w}_1}$.

  Let $J$ be the set of the first $n$ elements of $\dset{(j, k)}{j\in \N, 0 \le k \le b(j)}$ under the lexicographical order. We label the $n$ copies of $G$ in $\gsn$ by $J$, namely each vertex of $\gsn$ is an element of $V^J$, where $V = \sset{v_1, \dots, v_q}$ is the vertex set of $G$. Set $m := \max\dset{j}{(j,k)\in J}$.

  Our resolving set will be described by a matrix $S$ whose rows and columns are indexed by $\sset{0,1,\dots, m}$ and $J$ respectively with entries from $V$. Note that each row of $S$ is an element of $V^J$, thus can be seen as a vertex of $\gsn$. For $i \in \sset{0,1,\dots, m}$ and $(j,k) \in J$, we denote the entry of $S$ on row $i$ and column $(j,k)$ by $S(i,j,k) \in V$.

  We claim that a matrix $S$ can be chosen to satisfy the following properties.
  \begin{subequations}
    \begin{align}
        \sum_{i\preceq j} S(i,j,k)\mu(i,j) = r^k(w_1v_1 + \dots + w_qv_q), \quad \text{for all } (j,k)\in J; \label{pa} \\
        \sum_{i\preceq j} S(i,j',k)\mu(i,j) = 0, \quad \text{for all }(j',k) \in J\text{ and }j' < j \le m. \label{pb}
    \end{align}
  \end{subequations}
  We remark that \eqref{pa} and \eqref{pb} happen in the commutative ring $\Z[v_1, \dots, v_q]$ with unity.

  \begin{table}
    \centering
    {\renewcommand{\arraystretch}{1.2}
    \begin{tabular}{c|ccccccccccc}
     $i$ & 0 & 1 & 2 & 3 & 4 & 5 & 6 & 7 & 8 & 9 & \\
     \hline
     $(7,0)$ & $v_3$ & $v_3$ & $v_3$ & $v_2$ & $v_2$ & $v_1$ & $v_1$ & $v_1$ & $v_3$ & $v_3$ & \multirow{2}{*}{$\cdots$} \\
     $(7,1)$ & $v_1$ & $v_3$ & $v_3$ & $v_2$ & $v_2$ & $v_1$ & $v_1$ & $v_3$ & $v_1$ & $v_3$ \\
     \hline
     $\mu(i,j)$ & $-$ & $+$ & $+$ & $-$ & $+$ & $-$ & $-$ & $+$ & $0$ & $0$ &
    \end{tabular}}
    \caption{Values of $S(i,j,k)$ for $(j, k) = (7, 0), (7, 1)$ for $G = K_3$.}\label{values_of_sijk}
  \end{table}

  For example, when $G = K_3$, $q = 3$, we take $w = \begin{pmatrix}
    -1 & 0 & 1
  \end{pmatrix}^T$ and $r = 3$.
  In Table~\ref{values_of_sijk}, we supply the values of $S(i,j,k)$ for $(j, k) = (7, 0), (7, 1)$. The reader can verify \eqref{pa} in this case.

  In general, pick arbitrary $(j, k)\in J$. On the left hand side of \eqref{pa}, the summation consists of $2^{n(j)}$ terms, moreover $2^{n(j)-1}$ of them have $\mu(i,j) = +1$ (respectively $-1$). Since $r^k(w_1 + \dots + w_q) = 0$ and $r^k(\abs{w_1} + \dots + \abs{w_q}) = r^k\abs{w}_1 \le r^{b(j)}\abs{w}_1 \le 2^{n(j)}$ by \eqref{eqbj}, it is easy to assign one of $\sset{v_1, \dots, v_q}$ to $S(i,j,k)$ for all $i\preceq j$, possibly in many ways, to satisfy \eqref{pa}. For $i\npreceq j$, we take $S(i,j,k) = S(i\meet j, j, k)$. For every $(j', k)\in J$ and $j' < j \le m$, as $j \npreceq j'$, the left hand side of \eqref{pb} equals $\sum_{i\preceq j}S(i\wedge j',j',k)\mu(i,j) = 0$ by applying Lemma~\ref{mobius} to the function $f_{j',k}(i) = S(i, j', k)$.

  To show that $S$ resolves $\gsn$, it suffices to demonstrate that every $X\from J \to V$ is uniquely determined by the vector $$D := \left(\sum_{(j,k)\in J}d(X(j,k), S(i,j,k))\right)_{i=0}^m.$$ Suppose this vector $D = (D_0, \dots, D_m)$ is provided. We shall gradually uncover $\dset{X(j,k)}{0\le k \le b(j)}$ for $j = m, m-1, \dots, 0$. Assume that $\dset{X(j,k)}{0\le k \le b(j)}$ is known for every $j > j_0$. We extend the distance function $d\from V\times V\to \N$ of $G$ to the bilinear form \[
    d\left(\sum_{i=1}^q\alpha_iv_i, \sum_{i=1}^q\beta_iv_i\right) = \sum_{i=1}^q\sum_{j=1}^q\alpha_i\beta_jd(v_i, v_j),
  \]
  where $\alpha_1, \dots, \alpha_q$ and $\beta_1, \dots, \beta_q$ are in $\Q$. Observe that
  \begin{multline*}
    \sum_{i \preceq j_0}D_i\mu(i,j_0) = \sum_{i \preceq j_0}\left(\sum_{(j,k)\in J}d(X(j,k),S(i,j,k))\right)\mu(i,j_0) \\
    = \sum_{(j,k)\in J}d\left(X(j,k),\sum_{i \preceq j_0}S(i,j,k)\mu(i,j_0)\right) \stackrel{\eqref{pb}}{=} \sum_{j=j_0}^{m}\sum_{k = 0}^{b(j)}{d\left(X(j,k),\sum_{i \preceq j_0}S(i,j,k)\mu(i,j_0)\right)}.
  \end{multline*}
  Since both $(D_0, \dots, D_m)$ and $\dset{X(j,k)}{j_0 < j \le m, 0 \le k \le b(j)}$ are known, we are able to determine
  \begin{multline}\label{r-ary}
    \sum_{k=0}^{b(j_0)}d\left(X(j_0,k), \sum_{i\preceq j_0}S(i,j_0,k)\mu(i,j_0)\right) \stackrel{\eqref{pa}}{=} \sum_{k=0}^{b(j_0)}d\left(X(j_0,k), r^k\sum_{i=1}^qw_iv_i\right) \\
    = \sum_{k=0}^{b(j_0)}r^k\sum_{i=1}^q w_id(X(j_0,k),v_i) = \sum_{k=0}^{b(j_0)}r^k\sum_{i=1}^q M_{X(j_0,k),v_i}w_i = \sum_{k=0}^{b(j_0)}r^k\cdot(Mw)_{X(j_0,k)}.
  \end{multline}
  Let $g = \gcd_{i < j}((Mw)_i-(Mw)_j)$. We can thus deduce from \eqref{r-ary} the value of
  \begin{equation}\label{r-ary-2}
    \sum_{k=0}^{b(j_0)}r^k\cdot \frac{1}{g}\left((Mw)_{X(j_0,k)}-\min Mw\right).
  \end{equation}
  Notice that, according to our choice of $w$ and \eqref{r_condition}, $\left(\tfrac{1}{g}((Mw)_i - \min Mw)\right)_{i=1}^q$ are distinct integers in $[0, r)$. The value of \eqref{r-ary-2} uniquely decides $\dset{X(j_0, k)}{0 \le k \le b(j_0)}$.

  Finally, we estimate $m + 1$, the cardinality of the resolving set. Our choice of $m$ implies that $m$ is the smallest integer such that $\sum_{j=0}^m \max\sset{b(j)+1,0} \ge n$. For every $x \in \N$, by Theorem~\ref{estimate_ns},
  \begin{equation} \label{sum_of_bs}
      \sum_{j=0}^x \max\sset{b(j)+1,0} > \sum_{j=0}^x (n(j)\log_r 2 - \log_r\abs{w}_1) = \tfrac{1}{2}x\log_r x - O(x\log \log x).
  \end{equation}
  One can check that $x = 2n/\log_rn + O(n\log\log n/\log^2n)$ ensures the right hand side of \eqref{sum_of_bs} is $\ge n$.
\end{proof}

\section{A lower bound on $m(G, n)$}\label{lb}

A straightforward generalization of the lower bound on the coin weighing problem by Erd\H{o}s and R\'enyi gives a lower bound on the metric dimension of $\gsn$ (see Moser~\cite{MR0263643} and Pippenger~\cite{MR0437347} for different proofs using the second moment method and the information-theoretic method).

\begin{theorem}\label{lbt}
  Given a connected graph $G$ on $q \ge 2$ vertices, for every $n\in \N$, the metric dimension $m(G, n)$ of $\gsn$ is at least \[
    \left(2-\OO{\frac{\log\log n}{\log n}}\right)\frac{n}{\log_q n}.
  \]
\end{theorem}

\begin{proof}
  Let $S\subset V^n$ be a resolving set of $\gsn$ of size $m = m(G, n)$, where $V$ is the vertex set of $G$. We may assume without loss that $m = O(n)$. For every $s = (s_1, \dots, s_n)\in S$, let $X_1, \dots, X_n$ be independent random variables defined by $X_i = d(Y_i, s_i)$, where the independent random variables $Y_1, \dots, Y_n$ are chosen uniformly at random from $V$, and define
  \[
    A_s := \dset{(v_1, \dots, v_n)\in V^n}{\abs{\sum_{i=1}^n d(v_i, s_i) - \EE{\sum_{i=1}^nX_i}} < \sqrt{n\ln n}\cdot D},
  \]
  where $D$ is the diameter of the graph. Since each $X_i$ is bounded by $[0, D]$, Hoeffding's inequality provides an upper bound on the cardinality of the complement of $A_s$:
  \begin{multline*}
    \frac{\abs{V^n \setminus A_s}}{\abs{V^n}} = \pr{\abs{\sum_{i=1}^nd(Y_i, s_i)-\EE{\sum_{i=1}^nX_i}} \ge \sqrt{n\ln n}\cdot D} \\
    = \pr{\abs{\sum_{i=1}^nX_i-\EE{\sum_{i=1}^nX_i}} \ge \sqrt{n\ln n}\cdot D} \le 2\exp\left(-\frac{2\left(\sqrt{n\ln n}\cdot D\right)^2}{n\cdot D^2}\right) = \frac{2}{n^2}.
  \end{multline*}
  From the equivalent definition of a resolving set mentioned in Section~\ref{intro}, the function $d_S\from V^n\to\N^S$, defined by $(d_S(v_1, \dots, v_n))_s := d(v_1, s_1) + \dots + d(v_n,s_n)$ for every $(v_1, \dots, v_n)\in V^n$ and $s = (s_1, \dots, s_n) \in S$, is injective. Since the image of $\cap_{s\in S}A_s$ under $d_S$ is contained in a cube of side length $< 2\sqrt{n\ln n}\cdot D$ in $\N^S$, we obtain
  $$
    \left(2\sqrt{n\ln n}\cdot D\right)^m \ge \abs{\bigcap_{s\in S}A_s} \ge \abs{V^n} - \sum_{s\in S}\abs{V^n\setminus A_s} \ge q^n\left(1-\frac{2m}{n^2}\right) = q^n\left(1 - \OO{\frac{1}{n}}\right).
  $$
  Taking logarithm gives \[
    m \ge \frac{n\ln q - \OO{\frac{1}{n}}}{\frac{1}{2}\ln n + \OO{\log \log n}} = \frac{2n}{\log_q n}\cdot\frac{1-\OO{\frac{1}{n^2}}}{1 + \OO{\frac{\log\log n}{\log n}}} = \left(2 - \OO{\frac{\log\log n}{\log n}}\right)\frac{n}{\log_q n}.\qedhere
  \]
\end{proof}

\section{Asymptotically tight cases}\label{tight}

The bounds in Theorem~\ref{ubt} and Theorem~\ref{lbt} are asymptotically equivalent if and only if $r(G)$ defined by \eqref{parameter} equals $q$. We characterize the equality case.

\begin{lemma} \label{tfae}
  Given a connected graph $G$ on $q \ge 2$ vertices, let $M$ be the distance matrix of $G$. The following statements are equivalent.
  \begin{enumerate}[nosep]
    \item The technical parameter $r(G)$ defined by \eqref{parameter} equals $q$. \label{s1}
    \item There exists $w \in \Z^q$ such that $\sum_i w_i = 0$ and the vector $Mw$, after sorting its coordinates, is an arithmetic progression with nonzero common difference. \label{s2}
    \item There exists a permutation $\pi$ on $[q]$ such that \[
    \begin{pmatrix}
      \pi(1) \\
      \vdots \\
      \pi(q) \\
      0
    \end{pmatrix} \text{ is in the column space of }
    \begin{pmatrix}
      M & \one \\
      \one^T & 0
    \end{pmatrix},
  \] where the column space is understood as a subspace of $\Q^{q+1}$, and $\one$ is the $q$-dimensional all-ones column vector. \label{s3}
  \end{enumerate}
\end{lemma}

\begin{proof}
  Let $w\in \Z^q$ be a vector such that $\sum_{i=1}^q w_i = 0$ and the coordinates of $Mw$ are distinct integers, and let $g := \gcd_{i < j}((Mw)_i-(Mw)_j)$. Clearly, $\max Mw - \min Mw \ge (q-1)g$, and equality holds if and only if the vector $Mw$, after sorting its coordinates, is an arithmetic progression with common difference $g > 0$. This shows the implication from Statement \ref{s1} to Statement \ref{s2}. The converse is evident.

  Lastly, we demonstrate the equivalence between Statement \ref{s2} and Statement \ref{s3}. Suppose that there exists $w \in \Z^n$ such that $\one^T w_i = 0$ and the vector $Mw$, after sorting its coordinates, is an arithmetic progression with nonzero common difference. Thus there exists $a, b \in \Z$ with $b\neq 0$ and a permutation $\pi$ such that \[
    Mw = \begin{pmatrix}
      a + b\pi(1) \\
      \vdots \\
      a + b\pi(q)
    \end{pmatrix} = a\one + b\begin{pmatrix}
      \pi(1) \\
      \vdots \\
      \pi(q)
    \end{pmatrix}.
  \] We obtain that \[
    \begin{pmatrix}
      M & \one \\
      \one^T & 0
    \end{pmatrix}
    \begin{pmatrix}
      w \\
      -a
    \end{pmatrix} = \begin{pmatrix}
      Mw - a\one \\
      0
    \end{pmatrix} = b\begin{pmatrix}
      \pi(1) \\
      \vdots \\
      \pi(q) \\
      0
    \end{pmatrix},
  \] which implies Statement~\ref{s3}. Reversing the argument, one can show that Statement~\ref{s3} indicates the existence of $w\in \Q^q$ satisfying the conditions in Statement~\ref{s2}. However, one can always scale $w$ properly so that it becomes a vector in $\Z^q$.
\end{proof}

\begin{corollary}
  Given a connected graph $G$ on $q\ge 2$ vertices, let $M$ be the distance matrix of $G$. If $G$ is a complete graph, a path, a cycle or a complete bipartite graph, or the matrix $$M' := \begin{pmatrix}M & \one \\ \one^T & 0\end{pmatrix}$$ is invertible, then the metric dimension $m(G, n)$ of $\gsn$ is
  \[
    \left(2+\OO{\frac{\log\log n}{\log n}}\right)\frac{n}{\log_q n}.
  \]
\end{corollary}

\begin{proof}
  When $M'$ is invertible, Statement~\ref{s3} in Lemma~\ref{tfae} applies here. When $G$ is a complete graph, a path or a cycle, by Statement~\ref{s2} in Lemma~\ref{tfae}, it suffices to construct a vector $w\in \Z^q$ such that $\sum_i w_i = 0$ and the vector $Mw$, after sorting its coordinates, is an arithmetic progression with nonzero common difference. We list the construction of $w$ in Table~\ref{vector_construction} and leave the verification to the readers.

  \begin{table}
    \centering
    {\renewcommand{\arraystretch}{1.5}
    \begin{tabular}{cccc}
     complete graph & path & even cycle & odd cycle \\
     \hline
     $w_i = 2i-(q+1)$ & $w_i = \begin{cases}
       -1 & \text{if }i=1 \\
       1 & \text{if }i=q \\
       0 & \text{otherwise}
     \end{cases}$ & \rule{0pt}{50pt} $w_i = \begin{cases}
       +1 & \text{if }i=1 \\
       -\frac{q+2}{2} & \text{if }i=\frac{q}{2} \\
       \frac{q}{2} & \text{if }i=\frac{q+2}{2} \\
       0 & \text{otherwise}
     \end{cases}$ & $w_i = \begin{cases}
       \frac{q-3}{2} & \text{if }i=\frac{q+1}{2} \\
       \frac{q-1}{2} & \text{if }i=q \\
       -1 & \text{otherwise}
     \end{cases}$
    \end{tabular}
    \caption{Construction of $w \in \Z^q$ for complete graphs, paths and cycles.}\label{vector_construction}}
  \end{table}

  %



  Lastly, because $K_{2,2}$ is a cycle of length $4$, for a complete bipartite graph $G = K_{q_1, q_2}$, it suffices to check that $M'$ is invertible for $q_1 \neq 2$. Denote by $J_q$ the $q$-dimensional all-ones matrix, and by $I_q$ the $q$-dimensional identity matrix. Recall that $J_{q_1}$ has eigenvalues $0$ and $q_1$. As $q_1 \neq 2$, $J_{q_1} - 2I_{q_1}$ is invertible and $(J_{q_1} - 2I_{q_1})\one = (q_1-2)\one$, hence $\one^T(J_{q_1}-2I_{q_1})^{-1}\one = \frac{q_1^2}{q_1-2}$.
  Using row operations and Schur complements\footnote{Suppose $M = \begin{pmatrix}
    A & B \\ C & D
  \end{pmatrix}$ is a block matrix and $A$ is invertible. The Schur complement of the block $A$ is $M/A := D - CA^{-1}B$, which gives rise to the matrix equivalence $M \sim \begin{pmatrix}
    A & O \\ O & M/A
  \end{pmatrix}$.}, we have the following matrix equivalence:
  \begin{multline*}
    \begin{pmatrix}
      M & \one \\
      \one^T & 0
    \end{pmatrix} = \begin{pmatrix}
      2J_{q_1}-2I_{q_1} & J & \one \\
      J & 2J_{q_2}-2I_{q_2}& \one \\
      \one^T & \one^T & 0
    \end{pmatrix} \sim \begin{pmatrix}
      J_{q_1}-2I_{q_1} & O & \one \\
      O & J_{q_2}-2I_{q_2} & \one \\
      \one^T & \one^T & 0
    \end{pmatrix} \\
    \sim \begin{pmatrix}
      J_{q_1}-2I_{q_1} & O & \zero \\
      O & J_{q_2}-2I_{q_2} & \one \\
      \zero^T & \one^T & -\frac{q_1^2}{q_1-2}
    \end{pmatrix} \sim \begin{pmatrix}
      J_{q_1}-2I_{q_1} & O & \zero \\
      O & \left(1+\frac{q_1-2}{q_1^2}\right)J_{q_2}-2I_{q_2} & \zero \\
      \zero^T & \zero^T & -\frac{q_1^2}{q_1-2}
    \end{pmatrix}.
  \end{multline*}
  Notice that $\left(1+\frac{q_1-2}{q_1^2}\right)J_{q_2}-2I_{q_2}$ has eigenvalues $\left(1+\frac{q_1-2}{q_1^2}\right)q_2-2$ and $-2$, which are nonzero. Therefore $M'$ is invertible for a complete bipartite graph.
\end{proof}

\begin{remark}
  Seb\H{o} and Tannier~\cite[Section 1]{MR2065985} claimed that $m(P_q, n) \le (2+o(1))n/\log_q n$, where $P_q$ is the path on $q$ vertices, and they thought ``this upper bound is probably the asymptotically correct value''. Our result confirms their conjecture.
\end{remark}

\section{Open problems}\label{concluding}

Statement~\ref{s3} in Lemma~\ref{tfae} allows us to search for connected graphs $G$ on $q$ vertices with $r(G) > q$. For each connected graph $G$ on $q$ vertices, we check if the system of equations \[
  \begin{pmatrix}
    M & \one \\
    \one^T & 0
  \end{pmatrix}\begin{pmatrix}
    x_1 \\
    \vdots \\
    x_q \\
    x_{q+1}
  \end{pmatrix} = \begin{pmatrix}
    \pi(1) \\
    \vdots \\
    \pi(q) \\
    0
  \end{pmatrix}
\] has a solution for some permutation $\pi$ on $[q]$. Using McKay's dataset~\cite{mckay} of connected graphs on up to 10 vertices, we find 1 graph on 6 vertices, 4 graphs on 9 vertices and 1709 graphs on 10 vertices for which $r(G) > q$.

\begin{figure}[t]
  \centering
  \begin{tikzpicture}[thick, scale=0.5]
    \coordinate (1) at (90:2);
    \coordinate (2) at (150:2);
    \coordinate (3) at (210:2);
    \coordinate (4) at (270:2);
    \coordinate (5) at (330:2);
    \coordinate (6) at (30:2);
    \draw[darkgray] (1) -- (3);
    \draw[darkgray] (1) -- (4);
    \draw[darkgray] (1) node[vertex]{} -- (5);
    \draw[darkgray] (2) -- (3);
    \draw[darkgray] (2) -- (4);
    \draw[darkgray] (2) node[vertex]{} -- (5);
    \draw[darkgray] (6) -- (3);
    \draw[darkgray] (6) -- (4);
    \draw[darkgray] (6) node[vertex]{} -- (5);
    \draw[darkgray] (3) -- (5);
    \draw[darkgray] (3) node[vertex]{} -- (4);
    \draw[darkgray] (4) node[vertex]{} -- (5) node[vertex]{};
  \end{tikzpicture}\qquad
  \begin{tikzpicture}[thick, scale=0.5]
    \coordinate (0) at (140:2);
    \coordinate (1) at (40:2);
    \coordinate (2) at (-90:1.5);
    \coordinate (3) at (130:3);
    \coordinate (4) at (50:3);
    \coordinate (5) at (190:2);
    \coordinate (6) at (-10:2);
    \coordinate (7) at (90:2);
    \coordinate (8) at (0,0);
    \draw[darkgray] (8) -- (1) -- (7) -- (0) -- (5) -- (2) -- (6) -- (1);
    \draw[darkgray] (2) node[vertex]{} -- (8) -- (7) -- (3) -- (5) node[vertex]{} -- (8);
    \draw[darkgray] (1) node[vertex]{} -- (4) -- (6) node[vertex]{} -- (8) node[vertex]{} -- (0) node[vertex]{} -- (3) node[vertex]{};
    \draw[darkgray] (7) node[vertex]{} -- (4) node[vertex]{};
  \end{tikzpicture}\qquad
  \begin{tikzpicture}[thick, scale=0.46]
    \coordinate (0) at (90:2);
    \coordinate (1) at (150:2);
    \coordinate (2) at (30:2);
    \coordinate (3) at (0,0);
    \coordinate (4) at (0,-1);
    \coordinate (5) at (210:3);
    \coordinate (6) at (-30:3);
    \coordinate (7) at (-120:3);
    \coordinate (8) at (-60:3);
    \draw[darkgray] (0) -- (1) -- (7) -- (8) -- (2) -- (0) node[vertex]{} -- (3) -- (4) -- (7) -- (5) -- (4) -- (8) -- (6) -- (4) node[vertex]{};
    \draw[darkgray] (1) -- (2) node[vertex]{} -- (3) node[vertex]{} -- (1) node[vertex]{};
    \draw[darkgray] (5) node[vertex]{} -- (8) node[vertex]{};
    \draw[darkgray] (6) node[vertex]{} -- (7) node[vertex]{};
  \end{tikzpicture}\qquad
  \begin{tikzpicture}[thick, scale=0.44]
    \coordinate (1) at (90:3);
    \coordinate (2) at (120:3);
    \coordinate (3) at (60:3);
    \coordinate (4) at (90:1.5);
    \coordinate (5) at (150:2);
    \coordinate (6) at (30:2);
    \coordinate (7) at (210:2);
    \coordinate (8) at (-30:2);
    \coordinate (9) at (-90:2);
    \draw[darkgray] (1) -- (2) node[vertex]{} -- (5) -- (1);
    \draw[darkgray] (1) -- (3) node[vertex]{} -- (6) -- (1);
    \draw[darkgray] (4) -- (5) -- (7) -- (9) -- (8) -- (6) -- (4) -- (9);
    \draw[darkgray] (4) -- (7) -- (8) -- (4) node[vertex]{} -- (1) -- (7) node[vertex]{} -- (6);
    \draw[darkgray] (1) node[vertex]{} -- (8) node[vertex]{} -- (5) node[vertex]{} -- (9) node[vertex]{} -- (6) node[vertex]{};
  \end{tikzpicture}\qquad
  \begin{tikzpicture}[thick, scale=0.45]
    \coordinate (1) at (0:2);
    \coordinate (2) at (45:2);
    \coordinate (3) at (90:2);
    \coordinate (4) at (135:2);
    \coordinate (5) at (180:2);
    \coordinate (6) at (-135:2);
    \coordinate (7) at (-90:2);
    \coordinate (8) at (-45:2);
    \coordinate (9) at (90:3);
    \draw[darkgray] (1) -- (2) -- (3) -- (4) -- (5) -- (6) -- (7) -- (8) -- (1) -- (5) -- (2) -- (8) -- (6) -- (4) -- (2) node[vertex]{} -- (9) -- (4) node[vertex]{} -- (1) -- (6) node[vertex]{};
    \draw[darkgray] (8) node[vertex]{} -- (5) -- (7) -- (1) node[vertex]{} -- (3) -- (5) node[vertex]{};
    \draw[darkgray] (7) node[vertex]{} -- (3) node[vertex]{} -- (9) node[vertex]{};
  \end{tikzpicture}
  \caption{All the connected graphs on $2 \le q \le 9$ vertices with $r(G) > q$.}
\end{figure}
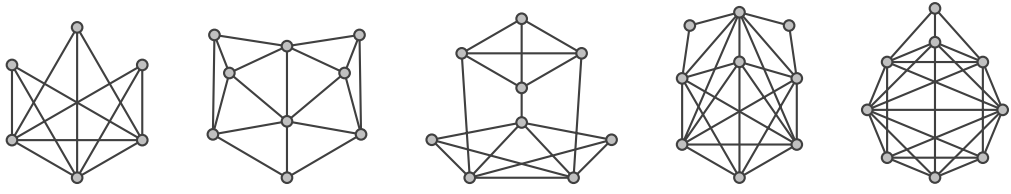

The graph on $6$ vertices is $K_6\setminus K_3$. We give a simple argument for $r(K_6\setminus K_3) = 7$ in Appendix~\ref{app_a}. We believe that our construction of a resolving set can be significantly improved for such graphs.

\begin{conjecture} \label{ca}
  Given a connected graph $G$ on $q \ge 2$ vertices, the metric dimension $m(G, n)$ of $\gsn$ is $(2+o(1))n/\log_q n$. In particular, $m(K_6\setminus K_3, n) = (2+o(1))n/\log_6 n$.
\end{conjecture}

In the proofs of Theorem~\ref{ubt} and Theorem~\ref{lbt}, we have made use of one property of graph distance, that is, it is integer valued. In addition, we used some other properties of graph distance in Remark~\ref{finite} just to show that Theorem~\ref{ubt} is not vacuously true for any connected graph. In this sense, our results are more related to integer matrices than graphs.

\begin{definition}
  Given a $p\times q$ integer matrix $M$ and $n\in \N$, $m(M, n)$ is the minimum cardinality of a subset $S$ of $[q]^n$ such that $M_S\from [p]^n \to \N^S$, defined by $(M_S(i_1, \dots, i_n))_s = M_{i_1,s_1} + \dots + M_{i_n,s_n}$ for every $(i_1, \dots i_n)\in [p]^n$ and $s = (s_1, \dots, s_n) \in S$, is an injection.
\end{definition}

If the difference between two rows of $M$, say the first two, is parallel to $\one^T$, then for every $n \ge 2$ and $S \subset [q]^n$, $M_S(1,2,1,1,\dots) = M_S(2,1,1,1,\dots)$, and so $m(M, n) = \infty$. Otherwise, Theorem~\ref{ubt} and Theorem~\ref{lbt} generalize to integer matrices naturally.

\begin{theorem} \label{btm}
  Given a $p \times q$ integer matrix $M$ with $p \ge 2$, if none of the differences between two rows of $M$ is parallel to $\one^T$, then for every $n\in\N$,
  \[
    \left(2-\OO{\frac{\log \log n}{\log n}}\right)\frac{n}{\log_p n} \le m(M, n) \le \left(2+\OO{\frac{\log \log n}{\log n}}\right)\frac{n}{\log_r n},
  \] where $r = r(M)$ is defined by \[
    r(M) = \min\dset{\frac{\max Mw - \min Mw}{\gcd_{i < j}((Mw)_i-(Mw)_j)}}{w\in \Z^q, \sum_i w_i = 0, (Mw)_i \neq (Mw)_j \text{ for all }i \neq j} + 1.
  \]
\end{theorem}

\begin{remark}
  The same argument in Remark~\ref{finite} shows that $p \le r(M) < \infty$ in Theorem~\ref{btm}.
\end{remark}

It is conceivable that the generalization of Conjecture~\ref{ca} to integer matrices holds.

\begin{conjecture}
  Given a $p\times q$ integer matrix $M$ with $p \ge 2$, if none of the differences between two rows of $M$ is parallel to $\one^T$, then $m(M,n) = (2+o(1))n/\log_p n$.
\end{conjecture}

\section*{Acknowledgements}
We are grateful to the referees for their close reading of this paper and many helpful comments.

\bibliographystyle{alpha}
\bibliography{similarity_code}

\appendix

\section{Proof of $r(K_6\setminus K_3) = 7$}\label{app_a}

\begin{proposition}
  Let $M$ be the distance matrix of $K_6\setminus K_3$. For every $w\in \Z^6$ such that $\sum_i w_i = 0$, the vector $Mw$, after sorting its coordinates, is never an arithmetic progression with nonzero common difference. Moreover, there exists $w\in\Z^6$ such that $\sum_i w_i = 0$ and the coordinates of $Mw$ consist of $0, 2, 3, 4, 5, 6$.
\end{proposition}

\begin{proof}
  Label the vertices of $K_6\setminus K_3$ of degree $3$ by $1,2,3$ and those of degree $5$ by $4,5,6$, and let $M$ be the distance matrix of $K_6\setminus K_3$:
  \[
    M := \begin{pmatrix}
      0 & 2 & 2 & 1 & 1 & 1 \\
      2 & 0 & 2 & 1 & 1 & 1 \\
      2 & 2 & 0 & 1 & 1 & 1 \\
      1 & 1 & 1 & 0 & 1 & 1 \\
      1 & 1 & 1 & 1 & 0 & 1 \\
      1 & 1 & 1 & 1 & 1 & 0
    \end{pmatrix}.
  \]
  Assume for the sake of contradiction that there exists $w\in \Z^6$ and a permutation $\pi$ on $[6]$ such that
  \begin{equation}\label{k6a}
    (Mw)_i = c + \pi(i)d \text{ for all }i\in[6],
  \end{equation}
  where $c \in \Z$ and $d \in \Z \setminus \sset{0}$. Observe that
  \begin{align*}
    3c + (\pi(1) + \pi(2) + \pi(3))d &= (Mw)_1 + (Mw)_2 + (Mw)_3 = 4(w_1 + w_2 + w_3) + 3(w_4 + w_5 + w_6),\\
    3c + (\pi(4) + \pi(5) + \pi(6))d &= (Mw)_4 + (Mw)_5 + (Mw)_6 = 3(w_1 + w_2 + w_3) + 2(w_4 + w_5 + w_6).
  \end{align*}
  Since $\sum_i w_i = 0$, we get that $\pi(1) + \pi(2) + \pi(3) = \pi(4) + \pi(5) + \pi(6)$, contradicting to $\sum_i \pi(i) = 15$.
  Finally, $w = \begin{pmatrix}
    5 & 3 & 2 & -2 & -3 & -5
  \end{pmatrix}^T$ satisfies $\sum_i w_i = 0$ and $Mw = \begin{pmatrix}
    0 & 4 & 6 & 2 & 3 & 5
  \end{pmatrix}^T$.
\end{proof}

\end{document}